\documentclass[12pt]{amsart}
\usepackage[colorlinks=true, pdfstartview=FitV, linkcolor=blue, citecolor=blue, urlcolor=blue]{hyperref}

\usepackage{amssymb,amsmath, amscd}
\usepackage{times, verbatim}
\usepackage{graphicx}
\usepackage[english]{babel}
 \usepackage[usenames, dvipsnames]{color}
\usepackage{amsmath,amssymb,amsfonts}
\usepackage{enumerate}

\usepackage{tikz}
\usetikzlibrary{calc, arrows, decorations.markings} \tikzset{>=latex}
\usepackage{amsmath,amssymb,amsfonts}
\usepackage{amscd, amssymb, latexsym, amsmath, amscd}
\usepackage[all]{xy}
\usepackage{pb-diagram}
\usepackage{enumerate}
\usepackage{ulem}
\usepackage{anysize}
\marginsize{2.5cm}{2.5cm}{2.5cm}{2.5cm}
\input xy
\xyoption{all}
\usepackage{pb-diagram}
\usepackage[all]{xy}

\usepackage{xcolor}
\input xy
\xyoption{all}

\theoremstyle{plain}
\newtheorem{thm}{Theorem}

\newtheorem{question}{Question}
\newtheorem{lemma}{Lemma}
\newtheorem{cor}{Corollary}

\newtheorem{prop}{Proposition}

\theoremstyle{definition} \theoremstyle{definition}
\newtheorem{remark}{Remark}
\newtheorem{example}{Example}

\theoremstyle{remark}

\newcommand{\G}{\textsc{\G}}

\newcommand{\A}{\mathbb{A}}
\newcommand{\Q}{\mathbb{Q}}

\newcommand{\Ad}{{\rm Ad}\,}

\newcommand{\Z}{\mathbb{Z}}

\newcommand{\C}{\mathbb{C}}

\newcommand{\F}{\mathbb{F}}

\def\G{{\rm G}}

\def\SL{{\rm SL}}

\def\GL{{\rm GL}}
\def\PGL{{\rm PGL}}

\def\Gal{{\rm Gal}}

\def\Sym{{\rm Sym}}

\def\Ad{{\rm Ad\, }}

\begin{document}


\title[Relations between cusp forms sharing Hecke eigenvalues]
{Relations between cusp forms sharing Hecke eigenvalues}

\begin{abstract}
  In this paper we consider the question of when the set of Hecke eigenvalues of a cusp form on $\GL_n(\A_F)$
  is contained in the set of Hecke eigenvalues of a cusp form on $\GL_m(\A_F)$ for $n \leq m$. This
  question is closely related to a question about finite dimensional representations of an
  abstract group, which also we consider in this work.
\end{abstract}

\author{Dipendra Prasad and Ravi Raghunathan}
\thanks{
  DP thanks  Science and Engineering research board of the
  Department of Science and Technology, India for its support
through the JC Bose
National Fellowship of the Govt. of India, project number JBR/2020/000006.
His work was also supported by a grant of
the Government of the Russian Federation
for the state support of scientific research carried out
under the  agreement 14.W03.31.0030 dated 15.02.2018. }

\address{Indian Institute of Technology Bombay, Powai, Mumbai-400076}
\address{DP: St Petersburg State University, St Petersburg, Russia}
\email{prasad.dipendra@gmail.com}
\email{ravir@math.iitb.ac.in}
\maketitle
    {\hfill \today}
    
\tableofcontents

\section{Introduction}
Let $F$ be a number field. Each automorphic representation $\pi$ of $\GL_d(\A_F)$
gives rise to Hecke eigenvalues (also called the Satake parameter), a $d$-tuple of (unordered) nonzero complex numbers
$H(\pi_v) =(a_{1v},\cdots, a_{dv})$,  at each place $v$ of $F$ where $\pi$ is unramified, and thus at almost all places of $F$.

Let $\pi_1$ and  $\pi_2$ be two  irreducible  automorphic representations of $\GL_n(\A_F)$ which are written as
isobaric sums:

\[\pi_1 = \pi_{11} \boxplus \pi_{12} \boxplus \cdots \boxplus \pi_{1\ell},\]
\[\pi_2 = \pi_{21} \boxplus \pi_{22} \boxplus \cdots \boxplus \pi_{2\ell'},\]
where $\pi_{1j}$ and $\pi_{2k}$ are irreducible  cuspidal automorphic representations of $\GL_{d_j}(\A_F)$ and 
$\GL_{d_k}(\A_F)$ respectively.
Then by the strong multiplicity one theorem due to Jacquet and Shalika cf. \cite{JS1}, \cite{JS2}, 
if $\pi_1$ and $\pi_2$  have the same Hecke eigenvalues $H(\pi_{1,v}) = H(\pi_{2,v})$ at almost all places $v$ of $F$ where $\pi_1,\pi_2$
are unramified, then $\ell = \ell'$, and up to a permutation of indices, $\pi_{1j} = \pi_{2j}$. 

In this paper, we will consider a variant of the strong multiplicity one theorem, identified in the following definition.

\vspace{10mm}

\noindent{\bf Definition:} Given automorphic representations $\pi_1$ on $\GL(n,\A_F)$ and $\pi_2$
on $\GL(m,\A_F)$, we say that $\pi_1$ is {\bf immersed} in $\pi_2$, written
$\pi_1 \preceq \pi_2$,  if the Hecke eigenvalues of $\pi_1$ (counted with multiplicity) are contained in the
Hecke eigenvalues of $\pi_2$ (counted with multiplicity) for almost all primes of the number field $F$.
On the other hand, we say that $\pi_1$ is {\bf embedded} in $\pi_2$, written as
$\pi_1 \subset \pi_2$ if there is an automorphic representation $\pi_3$ such that,
\[\pi_2 = \pi_1 \boxplus \pi_3.\]

\vspace{5mm}
The following is the motivating question for this paper.
\vspace{5mm}

\noindent{\bf Question:}

\begin{enumerate} \item Can it happen for distinct cuspidal representations that $\pi_1 \preceq \pi_2$?

\item If yes, can we classify all such pairs of cuspidal representations $\pi_1 \preceq \pi_2$?

  \end{enumerate}
\vspace{10mm}

One would have liked to assert that for cuspidal representations  $\pi_1 \preceq \pi_2$
 never happens if $n<m$, but that is not true. For
example, let $\pi$ be a cuspidal non-CM automorphic representation of $\PGL_2(\A_F)$.
At any unramified place $v$ of $F$, if $(a_v, a_v^{-1})$ are the Hecke eigenvalues of $\pi_v$, then
for the automorphic representation $\Sym^2(\pi)$ of $\PGL_3(\A_F)$, the Hecke eigenvalues  at the place $v$ of $F$, are
$(a_v^2,1,a_v^{-2})$. Thus the Hecke eigenvalues of the trivial representation of $\GL_1(\A_F)$ are contained
in the set of Hecke eigenvalues of the cuspidal automorphic representation $\Sym^2(\pi)$ of $\PGL_3(\A_F)$ at each unramified
place of $\pi$.

This paper is written in the hope that although $H(\pi_{1,v})$ can be contained in $H(\pi_{2,v})$ of $\pi_2$, at almost all the unramified places of $\pi_1$ and $\pi_2$, without $\pi_1$ being the same as $\pi_2$, this happens rarely, and only for
pairs of representations $(\pi_1,\pi_2)$ which are related in some well-defined way. 

We begin by proving the following proposition.

\begin{prop}
Let $\pi_1$ (resp. $\pi_2$) be an irreducible  cuspidal automorphic representation of $\GL_{n}(\A_F)$
(resp. $\GL_{n+1}(\A_F)$). 
Then $H(\pi_{1,v})$  of $\pi_1$ cannot be  contained in $H(\pi_{2,v})$ of $\pi_2$
at almost all places of $F$ where $\pi_1,\pi_2$ are unramified.
  \end{prop}

\begin{proof}
The proof is  a simple consequence of the strong multiplicity one theorem of Jacquet-Shalika recalled 
at the beginning of this paper. Let $\omega_1$ (resp. $\omega_2$) be the central character of 
$\pi_1$ (resp. $\pi_2$); these are Gr\"ossencharacters of $\GL_1(\A_F)$.
It is easy to see that, if $H(\pi_{1,v})$  is contained in $H(\pi_{2,v})$ at almost all places $v$ of $F$, then,
\[ \pi_1 \boxplus (\omega_2/\omega_1) = \pi_2,\]
which is not allowed by the strong multiplicity one theorem.
  \end{proof}

Here is another similarly `negative' result, this time proved with considerably more effort.

\begin{prop} \label{gl4}
  Let $\pi$ be an irreducible  cuspidal automorphic representation of
  $\GL_{4}(\A_F)$.
  Then $H(\pi_{v})$
cannot contain 1 at almost all places $v$ of $F$ where $\pi$ is unramified.
\end{prop}

\begin{proof}
  We will prove the proposition by contradiction, so assume that
  $H(\pi_{v})$ contains 1 at almost all places of $F$ where $\pi$ is unramified.
  Observe that to say that $H(\pi_{v})$ contains 1 is equivalent to saying that $\det(1- H(\pi_{v}))=0$,
  which translates into the following identity (assuming that $H(\pi_{v})$ operates on a 4 dimensional vector space $V$):
  \[ 1 - V + \Lambda^2(V) -\Lambda^3(V) + \Lambda^4(V) = 0.\]

  (One way to think of this identity is in the Grothendieck group of
  representations of an abstract group $G$ which comes equipped with a
  4-dimensional representation $V$ of $G$ such that the action of any $g \in G$ on $V$ has a nonzero fixed vector.)
  
Thus  we get the identity:
  \[ 1 + \Lambda^4(V)  +  \Lambda^2(V)  =  V + \Lambda^3 (V).\]

  Let the central character of $\pi$ be
  $\omega: \A_F^\times/F^\times \rightarrow \C^\times$. 
  Since we know by the work of Kim, cf. \cite{Kim}, that $\Lambda^2(\pi)$ is automorphic, by the strong multiplicity one theorem,
  we get an identity of the isobaric sum of automorphic representations:
      \[ 1 \boxplus \omega \boxplus  \Lambda^2(\pi)  =  \pi  \boxplus \omega\cdot \pi^\vee.\]
      Observe that the right hand side of this equality is a sum of two cuspidal representations on $\GL_4(\A_F)$, whereas
      there are two one dimensional characters of $\A_F^\times/F^\times$ on the left hand side. This is not allowed
      by the strong multiplicity one theorem,
      completing the proof of the proposition.
  \end{proof}

The following question lies at the basis of this work.

\begin{question} \label{conj}
Let $\pi_1$ (resp. $\pi_2$) be an irreducible cuspidal automorphic representation of $\GL_{n}(\A_F)$
(resp. $\GL_{n+2}(\A_F)$). 
Suppose that $H(\pi_{1,v})$ is contained in $H(\pi_{2,v})$ at almost all places of $F$ where $\pi_1$ and $\pi_2$ are unramified. Then is there an automorphic representation $\pi$ of $\GL_2(\A_F)$ with central character $\omega: \A_F^\times/F^\times \rightarrow \C^\times$,
and a character $\chi: \A_F^\times/F^\times \rightarrow \C^\times$,
such that,

\begin{eqnarray*} \pi_1 & = & \chi\cdot\omega \otimes \Sym^{n-1} (\pi),\\
\pi_2 & = & \chi \otimes \Sym^{n+1} (\pi),\end{eqnarray*}
i.e., up to twist by a character, is $(\pi_1,\pi_2) = (\omega   \otimes \Sym^{n-1} (\pi), \Sym^{n+1} (\pi))$?
  \end{question}
We will provide an affirmative answer to this question  for $n=1,2,3$ in this paper. On the other hand, in section \ref{finite}, we will
provide counter-examples to this question using the strong Artin conjecture
for all pairs of integers $(q-1,q+1)$ where $q\geq 5$ is a prime power.
The work \cite{Ca} proves strong Artin conjecture for certain cases for $q=5$, which allows us to  construct an unconditional counter example for the pair $(\GL_4, \GL_6)$ over $\Q$ in section \ref{finite}.

The question studied in this paper can be studied from a purely group theoretic point of view, and is discussed in section \ref{finite} from this perspective. We are unaware of this group theoretic point of view to have been put to use earlier; it seems of interest for
connected reductive groups too. In section \ref{conn}, we 
prove that the group theoretic question has an affirmative answer for groups which are
not virtually abelian, i.e., do not contain a subgroup of finite index which is abelian. Thus in the automorphic context, when one of the representations $\pi_1$, or $\pi_2$ is Steinberg
at a finite place, or has regular infinitesimal character at one of the archimedean places of $F$, our question should have an affirmative answer, but for the
moment, we do not know how to deal with it.

\begin{remark} Here is a geometric analogue of the questions being discussed in this paper.\footnote{This question has now been settled in a recent work of Khare and Larsen, cf. \cite{KL}.}
  Let $A$ and $B$ be abelian varieties
  over a number field $F$ with $A$ simple. For $v$ any finite place of $F$ where both $A$ and $B$ have good reduction, let $A_v, B_v$ denote their
  reductions mod $v$ (thus $A_v,B_v$ are abelian varieties over finite fields). Assume that there are isogenies from $A_v$ to $B_v$
  (not surjective as we are not assuming $\dim (A) = \dim(B)$) for almost all places $v$ of $F$ where $A$ and $B$ have good reduction.
  Then the question is if there is an isogeny from $A$ to $B$? If $\dim(A)=\dim(B)$, this is a consequence of
  the famous theorem of Faltings.
  \end{remark}

\begin{remark}
The paper was inspired by the notion of {\it relevance} introduced in \cite{GGP}, and to understand whether two global A-parameters which are locally relevant at all places must be globally relevant. This is not the case, and exactly for the reason discussed in this paper: that the Hecke eigenvalues of the cuspidal representation $\pi_1$ may be contained in the set of Hecke eigenvalues of the cuspidal representation $\pi_2$ at almost all places of the number field without $\pi_1$ being the same as 
$\pi_2$. 
\end{remark}

\begin{remark}
  Most of the paper deals with cusp forms $\pi_1$ on $\GL_n(\A_F)$, and $\pi_2$ on $\GL_m(\A_F)$
for the restricted pairs $(n,m)$ with $m=n+2$, as the first non-obvious case beyond $m=n$ and $m=n+1$, 
  such  that the Hecke eigenvalues
  of $\pi_1$ are contained in the Hecke eigenvalues of $\pi_2$ at almost all places of $F$.
  However, one might begin at the other
  extreme $(n,m)=(1,m)$, and try to classify cuspforms $\pi_2$ on $\GL_m(\A_F)$ such that the  Hecke eigenvalues of $\pi_2$ at almost all places of $F$ contain the eigenvalue 1. By Proposition \ref{case1} below, there is a nice answer for $(n,m)=(1,3)$, whereas by Proposition \ref{gl4}, there are none in the case $(n,m)=(1,4)$.
    It is easy to see that 
 the cuspidal representations $\pi_2$ of $\GL_6(\A_F)$ which arise as
  the
  basechange of a cuspidal selfdual
  representation of $\PGL_3(\A_E)$ for $E/F$ quadratic,  have Hecke eigenvalue 1 at almost all places of $F$.
  Using Theorem 2 of the paper \cite{Yam} and  using a similar identity as in the proof of Proposition  \ref{gl4}
  which this time would be:
  \[ 1 + \Lambda^2+  \Lambda^4(V) +\Lambda^6(V)  =  V + \Lambda^3 (V) + \Lambda^5(V),\]
   a cuspidal representations $\pi_2$ of $\GL_6(\A_F)$
having Hecke eigenvalue 1 at almost all places of $F$ arises as
  the
  basechange of a cuspidal 
  representation of $\GL_3(\A_E)$ for $E/F$ quadratic
  (which is most likely selfdual, and on $\PGL_3(\A_F)$, but this we have not proved). 
    We have not investigated the situation for general pairs $(1,m)$. 
\end{remark}

We end the introduction by remarking that the last two sections of the paper, sections 5 and 6, are written for finite groups and
Lie groups respectively, and are quite independent of the earlier sections. Section 5 eventually has implications for automorphic representations
through the known cases of (strong) Artin's conjecture. Since automorphic representations on $\GL_n(\A_F)$, $F$ a function field,
are characterised by their Galois representations by the work of L. Lafforgue, these sections also construct both counter-examples to
Question \ref{conj}  above, and to assert that the only counter-examples come from potentially abelian automorphic representations.

\vspace{1cm}

\noindent{\bf Acknowledgements:} The authors thank F. Calegari and D. Ramakrishnan for helpful correspondences, and the referee  for  careful reading of our paper and several useful comments.
\section{Preliminaries}

For automorphic representations $\pi_1, \pi_2$ on $\GL_a(\A_F), \GL_b(\A_F)$,
we denote by $\pi_1\boxplus \pi_2$
the isobaric sum of $\pi_1,\pi_2$, which is an
automorphic form on $\GL_{a+b}(\A_F)$. If $H(\pi_{1,v})$ and  $H(\pi_{2,v})$
are the
Hecke eigenvalues of $\pi_1$ and $\pi_2$, then the Hecke eigenvalues of 
$\pi_1\boxplus \pi_2$ is the union (with multiplicities)
of $H(\pi_{1,v}), H(\pi_{2,v})$.

We will also use the notation  $A \boxplus B$ where
$A$ (resp. $B$)
is any collection of $a $ (resp.  $b$)
nonzero complex numbers  defined for almost all finite places $v$ of $F$.
In this
generality, we will use partial $L$ function $L(s,A)$, where the Euler product
is taken outside a finite set $S$ of places, $S$ containing
the places at infinity.  

In the same spirit, for $A,B$ as in the last para,
we will define $A \boxtimes B$ to be a collection of $a\cdot b $
nonzero complex numbers  for almost all finite places $v$ of $F$, and
the associated partial Rankin product $L$ function $L(s,A \boxtimes B)$,
again where the Euler product
is taken outside a finite set $S$ of places, $S$ containing
the places at infinity.

\begin{lemma} \label{lemma0}
  Suppose $C$, $D$ are automorphic representations on $\GL_c(\A_F), \GL_d(\A_F)$, and $\chi$ a Gr\"ossencharacter on $F$.
  Suppose $A$ 
is any collection of $c+d-1 $ 
nonzero complex numbers defined for almost all finite places $v$ of $F$ such that
\[ A \boxplus \chi = C \boxplus D.\]
Suppose that $L(\chi^{-1}A, s)$ is known to have meromorphic continuation to the entire complex plane with no zero at $s=1$. Then
there is an automorphic representation $\pi_3$ of $\GL_{c+d-1}(\A_F)$ whose Hecke eigenvalues are  equal
to $A$ at almost all finite places $v$ of $F$.
  \end{lemma}

\begin{proof}
  Expand
  $\chi^{-1}(C \boxplus D)$ as an isobaric sum of cusp forms, and
  note that for any cusp form $\pi$ on $\GL_m(\A_F)$,
   $L(\pi,s)$, the partial $L$-function without regard to omitted set of places, has a pole at $s=1$ if and only if $m=1$, and $\pi = 1$. Therefore by what is given
  for $L(\chi^{-1}A, s)$, the isobaric sum decomposition of
  $\chi^{-1}(C \boxplus D)$ in terms of cusp forms
  must contain the trivial representation of $\GL_1(\A_F)$, omitting which from $\chi^{-1}(C \boxplus D)$ defines $\chi^{-1}A$ as an automorphic
  representation. 
  \end{proof}

\begin{lemma}\label{lemma1}
Suppose $\pi_1$ is a cuspidal automorphic representation on $\GL_{n}(\A_F)$ and $\pi_2$ is a cuspidal automorphic representation on $\GL_{n+2}(\A_F)$ such 
that at almost all unramified places of $\pi_1$ and $\pi_2$,
$H(\pi_{1,v})\subseteq H(\pi_{2,v})$.
Let $\omega_1$ (resp. $\omega_2$) be the central character of 
$\pi_1$ (resp. $\pi_2$) which is a Gr\"ossencharacter on $\GL_1(\A_F)$. Suppose that 
$\Lambda^2(\pi_2),  \Sym^2(\pi_1)$ are known to be automorphic.
Then, 
\begin{enumerate}
\item    The Rankin product $\pi_1 \boxtimes \pi_2$ is automorphic.
\item We have the isobaric decomposition of automorphic representations:
\[  \pi_2 \boxtimes \pi_1 \boxplus \omega_2/ \omega_1 = \Lambda^2(\pi_2) \boxplus \Sym^2(\pi_1).\]
\end{enumerate}
 \end{lemma}

\begin{proof}
    We first prove the identity
  expressed in (2), i.e., 
    that the two sides
  have the same Hecke eigenvalues at almost all the primes of $F$. This task is made more transparent by looking at vector
  spaces $V,W, A$ with $V=W+ A$ with $A$ two dimensional, and noting the identity:
 
  \begin{eqnarray*} V \otimes W + \Lambda^2 (A) & = & W \otimes W + A \otimes W+ \Lambda^2 (A)  \\
    & = & \Lambda^2(W)  + A \otimes W +\Lambda^2 (A)  + \Sym^2(W) \\
    & = & \Lambda^2(V) +  \Sym^2(W).
  \end{eqnarray*}

  Now, Lemma \ref{lemma0} proves the automorphy of $\pi_2 \boxtimes \pi_1$ since its $L$-function is known to be entire and non-vanishing
  on the line Re$(s) =1$ by the Rankin-Selberg theory, see Theorem 5.2 of Shahidi \cite{Sha}. \end{proof}

\begin{lemma}\label{lemma2}
  Suppose $\pi_1, \pi_2,\pi_3$ are cuspidal automorphic representations on $\GL_{n_i}(\A_F)$ (for $i=1,2,3$).
  Suppose that the Rankin products $\pi_1 \boxtimes \pi_2$ and $\pi_2 \boxtimes \pi_3$ are known to be automorphic.
  Then in the isobaric sum decomposition of $\pi_1 \boxtimes \pi_2$, $\pi_3^\vee \subset \pi_1 \boxtimes \pi_2$ if and only if
  in the isobaric sum decomposition of $\pi_2 \boxtimes \pi_3$, $\pi_1^\vee \subset \pi_2 \boxtimes \pi_3$.
\end{lemma}
\begin{proof}
  Since the Rankin products $\pi_1 \boxtimes \pi_2$ is given to be automorphic,  by Jacquet-Shalika, the
  L-function,
  \[L(s, \pi_1 \boxtimes \pi_2 \boxtimes \pi_3),\]
  has a pole at $s =1$ if and only if $\pi_3^\vee \subset \pi_1 \boxtimes \pi_2$. Same triple-product L-function
  dictates $\pi_1^\vee \subset \pi_2 \boxtimes \pi_3$.
  \end{proof}

Besides the strong multiplicity one theorem of Jacquet-Shalika, we will use the symmetric square lift of Gelbart-Jacquet cf. \cite{GJ} which we state in the form we will use. The work of Gelbart-Jacquet was to establish the symmetric square lift from $\GL_2$ to
$\GL_3$; for characterising the image of the
symmetric square lift,  we refer to \cite{Ra1}.

\begin{thm} \label{GJ}
  Let $\pi_2$ be a cuspidal automorphic representation of $\PGL_3(\A_F)$ with $\pi_2 \cong \pi_2^\vee$.
  Then $\pi_2$ arises as the adjoint lift of an automorphic representation $\pi$ of
  $\GL_2(\A_F)$, i.e., 
  \[ \pi_2 \cong \Ad( \pi) = \omega^{-1} \otimes \Sym^2(\pi),\]
  where $\omega$ is the central character of $\pi$, a  Gr\"ossencharacter on $\GL_1(\A_F)$. 
  \end{thm}

\begin{cor} \label {cor1}
  Let $\pi_2$ be a cuspidal automorphic representation of $\GL_3(\A_F)$ with $\pi_2 \cong \chi \otimes \pi_2^\vee$ for $\chi$
  a  Gr\"ossencharacter on $\GL_1(\A_F)$. 
  Then $\pi_2$ can be written as 
  \[ \pi_2 =  \lambda \otimes \Sym^2(\pi),\]
  where $\pi$ is a cuspidal automorphic representation of $\GL_2(\A_F)$, and $\lambda$ a  Gr\"ossencharacter on $\GL_1(\A_F)$. 
\end{cor}

\begin{proof} Let $\omega_2$ be the central character of $\pi_2$. Comparing the central characters for the given isomorphism:
  \[\pi_2 \cong \chi \otimes \pi_2^\vee, \tag{1}\]
  it follows that,
  \[ \omega_2^2 = \chi^3.\]
  Therefore for $\mu = \chi/\omega_2$, the isomorphism in (1) can be rewritten as:
  \[\pi_2 \cong \mu^{-2} \otimes \pi_2^\vee, \tag{2}\]
  or,
  \[(\mu \otimes \pi_2) \cong (\mu \otimes \pi_2)^\vee. \tag{3}\]

  Therefore, the representation $\mu \otimes \pi_2$ of $\GL_3(\A_F)$ is selfdual.
  Comparing the central characters
  on the two sides of equation (3), we find that the central character $\omega$ of the representation $\mu \otimes \pi_2$ of $\GL_3(\A_F)$ is quadratic.
  Twisting the representation $\mu \otimes \pi_2$ of $\GL_3(\A_F)$ by $\omega$,
  we find that the representation $(\omega\mu) \otimes \pi_2$ of $\GL_3(\A_F)$ is both selfdual
  and of trivial central character, so Theorem \ref{GJ} applies, proving that $\pi_2$ is a symmetric square up to a twist.
  \end{proof}

\section{The results}
We introduce the following notation keeping 
Question \ref{conj} in mind.
Suppose $\pi_1$ is a cuspidal automorphic representation on $\GL_{n}(\A_F)$ and $\pi_2$ is a cuspidal automorphic representation on $\GL_{m}(\A_F)$ such 
that at almost all unramified places of $\pi_1$ and $\pi_2$,
$H(\pi_{1,v})\subseteq H(\pi_{2,v})$, we write $\pi_1 \preceq \pi_2$.

We observe that we may twist the pair $(\pi_1,\pi_2)$ appearing in Question \ref{conj} by a Gr\"ossencharacter. 
Accordingly, in the following proposition that provides an affirmative answer to
Question \ref{conj} for $n=1$ we may assume that $\pi_1=1$.

\begin{prop} \label{case1}
Let $1$ denote the trivial representation of $\GL_1(\A_F)$ and suppose that $\pi_2$ is a cuspidal automorphic representation on $\GL_3(\A_F)$ such that $1\preceq\pi_2$. Then $\pi_2$ is a self-dual representation of $\PGL_3(\A_F)$, and   arises as
  $\omega^{-1}\cdot \Sym^2(\pi) $ (the adjoint lift)
  of   a cuspidal automorphic form $\pi$ on $\GL_2(\A_F)$ with central character
  $\omega: \A_F^\times/F^\times \rightarrow \C^\times$.
  \end{prop}
\begin{proof}
  The proof will be a simple consequence of the strong multiplicity one theorem of Jacquet-Shalika recalled in the beginning of this paper
  and Theorem \ref{GJ} due to Gelbart-Jacquet. 
  Let $\omega_2$ be the central character of $\pi_2$ which is a Gr\"ossencharacter on $\GL_1(\A_F)$.
By Lemma \ref{lemma1} in this case for $(\pi_1,\pi_2)=(1, \pi_2)$, it follows that
     \[ \Lambda^2(\pi_2) \boxplus 1 = \pi_2 \boxplus \omega_2.\]

     Therefore, by the strong multiplicity one theorem, we deduce that:

     \begin{enumerate}
\item     $\omega_2 = 1$,
\item $\Lambda^2(\pi_2) = \pi_2$. \end{enumerate}
     
     By (1) and (2), we find that
     \[\pi_2 \cong \pi_2^\vee.\]
     Therefore, by Theorem \ref{GJ} due to Gelbart-Jacquet,
     $\pi_2$ arises as the adjoint  lift from
     a cuspidal automorphic form $\pi$ on $\GL_2(\A_F)$, i.e., \[\pi_2 = \omega^{-1}\cdot \Sym^2(\pi),\]
proving the proposition.     \end{proof}

The following proposition provides an affirmative answer to Question \ref{conj} for $n=2$.

\begin{prop} \label{case2}
  Suppose that $\pi_1$ is a cuspidal  automorphic representation on $\GL_2(\A_F)$ with central
  character $\omega_1: \A_F^\times/F^\times \rightarrow \C^\times$,
  and that $\pi_2$ is a cuspidal automorphic representation on $\GL_4(\A_F)$, and that
  $\pi_1\preceq \pi_2$. Then,
\begin{enumerate}
\item $\pi_1$ cannot be CM (a CM representation is one defined using a Gr\"ossencharacter on a quadratic extension $E$ of $F$).

\item $\pi_2 = \omega_1^{-1} \otimes \Sym^{3} (\pi_1)$.
  \end{enumerate}

  \end{prop}
\begin{proof}
    Let $\omega_2$ be the central character of $\pi_2$ which is a Gr\"ossencharacter on $\GL_1(\A_F)$.
  By Lemma \ref{lemma1},
  \[ \pi_2 \boxtimes \pi_1
  \boxplus \omega_2/ \omega_1 = \Lambda^2(\pi_2)
  \boxplus \Sym^2(\pi_1)
  \tag{1},\] 
  where all the terms appearing above are automorphic:  $\Lambda^2(\pi_2)$ by Kim \cite{Kim}, $\Sym^2(\pi_1)$ by Gelbart-Jacquet \cite{GJ},
  and $\pi_2 \boxtimes \pi_1$
  by Lemma \ref{lemma1}.

  We first assume that  $\pi_1$ is CM. Observe that since $\pi_1$ is a cusp form on
  $\GL_2(\A_F)$ and $\pi_2$ a cusp form on $\GL_4(\A_F)$, $\pi_2 \boxtimes \pi_1$ cannot contain any Gr\"ossencharacter.
  Therefore, by the isobaric decomposition (1), exactly one of the two terms $\Lambda^2(\pi_2)$ or  $\Sym^2(\pi_1)$ may contain a Gr\"ossencharacter.
  Since we have assumed that $\pi_1$ is CM, $\Sym^2(\pi_1)$ contains a Gr\"ossencharacter, and therefore, 
  $\Lambda^2(\pi_2)$ cannot contain a Gr\"ossencharacter if $\pi_1$ is CM.

  Since $\pi_1$ is CM, we can write,
  \[ \Sym^2(\pi_1) =\pi_3 \boxplus \chi_3,\]
  where $\chi_3$ is  a Gr\"ossencharacter, and $\pi_3$ must be cuspidal (because the left hand side of (1) has only one Gr\"ossencharacter
  in its isobaric decomposition). By (1), $\chi_3 = \omega_2/\omega_1$, and we can simplify (1) to
\[ \pi_2 \boxtimes \pi_1 = \Lambda^2(\pi_2)
  \boxplus \pi_3.\]
 From  Lemma \ref{lemma2}, it follows that,
  \[ \pi_2 = \pi_1 \boxtimes \pi_3^\vee.\]
  Therefore,
  \[ \Lambda^2(\pi_2)
  = [\Sym^2(\pi_1) \boxtimes \Lambda^2(\pi_3^\vee)]
  \boxplus [\Lambda^2(\pi_1) \boxtimes \Sym^2(\pi_3^\vee)]  .\]

  Since  $\Sym^2(\pi_1) =\pi_3\boxplus \chi_3$, therefore as
$\Lambda^2(\pi_3^\vee)$ is a Gr\"ossencharacter, we find that $\Lambda^2(\pi_2)$
  contains a Gr\"ossencharacter which is not allowed, proving that $\pi_1$ cannot be a CM form.

  Now we turn to the case when $\pi_1$ is not CM in which case it is known by Gelbart-Jacquet that
  $\Sym^2(\pi_1)$ is a cuspidal automorphic
    representation of $\GL_3(\A_F)$. Therefore by the strong multiplicity one theorem applied to (1),
    we make the following conclusions:

    \begin{enumerate}
      \item the character
        $\omega_2/ \omega_1: \A_F^\times/F^\times \rightarrow \C^\times$
        must belong to the isobaric sum decomposition of
        $\Lambda^2(\pi_2)$, in particular, $\pi_2 \cong \pi^\vee_2 \otimes (\omega_2/\omega_1)$, i.e.,
        $\pi_2$ has parameter in the symplectic similitude group, and considering the similitude factor, we find:
        \[ (\omega_2/\omega_1)^2 = \omega_2, \;\;\;  \text { i.e., } \omega_2 = \omega_1^2. \tag{2}\]

      \item
        $\Sym^2(\pi_1)$ must be contained in  the isobaric sum decomposition of
        $\pi_2 \boxtimes \pi_1$.
        \end{enumerate}

    Since by \cite{KS1}, \cite {KS2},    $\pi_2 \boxtimes \pi_1$ and    $\Sym^2(\pi_1) \boxtimes \pi_1$ are known to be
    automorphic, we can apply Lemma \ref{lemma2}, and conclude that:
    \[ \pi_2^\vee =   \pi_2 \otimes (\omega_1/\omega_2) \subset \pi_1 \boxtimes \Sym^2(\pi_1^\vee) = \pi_1 \boxtimes \omega_1^{-2} \boxtimes \Sym^2(\pi_1)   
    \tag{3}.\]
    It is easy to see that,
     \[\pi_1 \boxtimes \Sym^2(\pi_1) = (\omega_1 \otimes \pi_1) \boxplus  \Sym^3(\pi_1),\]
     therefore we can write (3) as:

     \[\pi_2 \otimes (\omega_1/\omega_2) \subset  (\omega_1^{-1} \otimes \pi_1) \boxplus \omega_1^{-2} \boxtimes \Sym^3(\pi_1). \tag{4} \]   
     Since $\pi_2$ is a cuspidal automorphic representation of $\GL_4(\A_F)$, 
 applying  the strong multiplicity one theorem to (4), the only option we have (after using (2)) is that:
       \[\pi_2 =  \omega_1^{-1} \otimes \Sym^3(\pi_1),\]
       proving the proposition.
\end{proof}

The following proposition provides an affirmative answer to  Question \ref{conj} for $n=3$. The proof of this proposition will use
the unproved cases of functorialty for $\Lambda^2(\pi_2)$ where $\pi_2$ is a cuspform on $\GL_5(\A_F)$, as well as $\Sym^6(\pi)$ for $\pi$ a cusp form
on $\GL_2(\A_F)$. It may be mentioned that
although automorphy  
of $\Sym^2(\pi_1)$ for $\pi_1$ a cuspform on $\GL_3(\A_F)$ is not known, in our context below, it will be applied to $\pi_1$
which is selfdual up to a twist, and hence is a symmetric square of a cuspform on $\GL_2(\A_F)$ up to a twist by Corollary
\ref{cor1} which allows one to conclude automorphy of $\Sym^2(\pi_1)$ using known cases of functoriality for $\Sym^4(\pi)$.

\begin{prop}\label{case3}
Suppose that $\pi_1$ is a cuspidal automorphic representation on $\GL_3(\A_F)$ and that $\pi_2$ is a cuspidal automorphic representation on $\GL_5(\A_F)$ such that $\pi_1\preceq\pi_2$. 
Then, there exists a cuspidal automorphic representation $\pi$ of $\GL_2(\A_F)$ of central character $\omega$
such that up to simultaneous twisting of the pair $(\pi_1,\pi_2)$ by a Gr\"ossencharacter, we have:
\begin{eqnarray*} \pi_1 & = & 
    \Sym^{2} (\pi),\\
\pi_2 & = & \omega^{-1} \otimes \Sym^{4} (\pi).\end{eqnarray*}
  \end{prop}
\begin{proof}
    Let $\omega_1$ (resp. $\omega_2$) be the central character of $\pi_1$ (resp. $\pi_2$) which is a Gr\"ossencharacter on $\GL_1(\A_F)$.
    By Lemma \ref{lemma1}, for any Gr\"ossencharacter $\chi$ on $F$:
\[  L(s, \pi_2 \times \pi_1 \times \chi) L(s, \omega_2/ \omega_1 \times \chi ) = L(s, [\Lambda^2(\pi_2) \oplus\Sym^2(\pi_1)]\times \chi), \tag{1} \]

  Since  $\pi_2$ is a cuspidal automorphic representation of $\GL_5(\A_F)$,
 it is known by Jacquet-Shalika, cf. \cite{JS}, that  $L(s, \Lambda^2(\pi_2))$
 cannot have a pole at $s=1$. Therefore
  for   $\chi=\omega_1/ \omega_2 $,  since the left hand side of the  product of $L$-functions  in (1) has a simple pole at $s=1$,
  right hand side of (1) too must have a simple pole, contributed therefore
  by  $L(s,\Sym^2(\pi_1) \otimes \omega_1/\omega_2)$. 
  In particular,
  $\pi_1^\vee \cong \pi_1 \boxtimes \omega_1/\omega_2$. By Corollary \ref{cor1}, such representations of $\GL_3(\A_F)$ arise as a twist of a symmetric square:
  \[ \pi_1 \cong \lambda \otimes \Sym^2(\pi),\]
  for a cuspidal automorphic representation $\pi$ of $\GL_2(\A_F)$ of central character $\omega$, and $\lambda$ a Gr\"ossencharacter on $\A_F^\times$. Twisting the pair $(\pi_1,\pi_2)$ by
  $\lambda^{-1}$, we assume that $\pi_1 \cong  \Sym^2(\pi)$.

  Since $\pi_1 = \Sym^2(\pi)$, and it is easy to see that,
    \[ \Sym^2(\pi_1)= \Sym^2(\Sym^2(\pi)) = \omega^2 + \Sym^4(\pi),\] 
    and therefore  by Kim, cf. \cite{Kim}, since $\Sym^4(\pi)$ is known to be automorphic, so is
    $\Sym^2(\pi_1)$. Since we are assuming that
    $\Lambda^2(\pi_2)$ is known to be automorphic,  Lemma \ref{lemma1} applies,
    allowing us to conclude that $\pi_2 \boxtimes \pi_1$ is automorphic and we have the
    isobaric decomposition:
    \[ \pi_2 \boxtimes \pi_1
    \boxplus \omega_2/ \omega_1 = \Lambda^2(\pi_2)
  \boxplus \Sym^2(\pi_1)
  \tag{2}.\] 
  
  Therefore,  by the strong multiplicity one theorem applied to (2),
    we conclude
        $\Sym^2(\pi_1)$ must be contained in  the isobaric sum decomposition of
        $\pi_2 \boxtimes \pi_1$ as a direct summand.

     Applying Lemma \ref{lemma2} to (2), we conclude that:
    \[ \pi_2 \subset \pi_1^\vee \boxtimes \Sym^2(\pi_1) \tag{3},\]
    again as a direct summand, since as we will see now, $\pi_1^\vee \boxtimes \Sym^2(\pi_1)$ is automorphic by
    our assumption that $\Sym^6(\pi)$ is automorphic.

    Since $\pi_1 = \Sym^2(\pi)$, and $ \Sym^2(\Sym^2(\pi)) = \omega^2 + \Sym^4(\pi),$ 
    we find that:
    \begin{eqnarray*} \pi_1 \boxtimes \Sym^2(\pi_1) & = & \Sym^2(\pi) \boxtimes (\omega^2 \boxplus \Sym^4(\pi)),\\ 
      & = & \omega^2 \Sym^2(\pi)   \boxplus  \Sym^2(\pi) \boxtimes \Sym^4(\pi)   ,\\
      & = & \omega^2 \Sym^2(\pi)\boxplus \omega^2 \Sym^2(\pi) \boxplus \omega \Sym^4(\pi) \boxplus \Sym^6(\pi),
      \end{eqnarray*}
    if particular, if $\Sym^6(\pi)$ is automorphic, so is $\pi_1^\vee \boxtimes \Sym^2(\pi_1)$.

    Since  $\pi_2 \subset \pi_1^\vee \boxtimes \Sym^2(\pi_1) = \omega^{-2} \pi_1 \boxtimes \Sym^2(\pi_1),$
    we find that:
    \[ \pi_2 \subset  \Sym^2(\pi)\boxplus  \Sym^2(\pi) \boxplus \omega^{-1} \Sym^4(\pi) \boxplus \omega^{-2}\Sym^6(\pi). \tag{4}\]
    Now $\pi_2$ is a cuspidal representation on $\GL_5(\A_F)$, and by Proposition \ref{sym6} proved in the next section,
    isobaric decomposition of $\Sym^6(\pi)$ cannot have a cuspidal representation of $\GL_5(\A_F)$. 
    Therefore applying  the strong multiplicity one theorem to (4), we find that the only option we have is that $\Sym^4(\pi)$
    is cuspidal, and 
            \[\pi_2 = \omega^{-1}\Sym^4(\pi),\]
       proving the proposition.
\end{proof}

\vskip 5pt

\begin{remark} \label{inv}
The identity proved in Lemma \ref{lemma1}:
  \[  \pi_2 \otimes \pi_1 +  \omega_2/ \omega_1 = \Lambda^2(\pi_2) + \Sym^2(\pi_1), \tag{1}\]
  holds, as the proof shows,  among any two representations $(\pi_1,V_1)$ and $(\pi_2,V_2)$
  of an abstract group $G$ when $\dim(\pi_2) - \dim(\pi_1)=2$ and when
  for any $g \in G$, the set of eigenvalues
  of $\pi_1(g)$ acting on $V_1$, counted with multiplicity, is contained in the set of eigenvalues of $\pi_2(g)$ acting on $V_2$, counted with multiplicity.
  All the proofs in this section of Propositions \ref{case1}, \ref {case2}, \ref{case3} giving an affirmative answer to Question \ref{conj} for the pair $(\GL_n,\GL_{n+2})$ for $n=1,2,3$ use this identity (1) crucially,
  answering Question \ref{conj}
  for any group $G$, and then we had to carefully transport that proof (for arbitrary group $G$)
  to the world of automorphic forms using the
  strong multiplicity one theorem about isobaric decomposition of automorphic forms, and instances of functoriality.  However,
  when in section \ref{conn}, we answer Question \ref{conj} in the affirmative for general groups which are not
  virtually abelian, we do not rely on the identity (1).
    \end{remark}

\begin{remark}(Local-global compatibility of automorphic representations) The identities among Hecke
  eigenvalues used in this paper, such as 
in Lemma \ref{lemma1}:
\[  \pi_2 \otimes \pi_1 +  \omega_2/ \omega_1 = \Lambda^2(\pi_2) + \Sym^2(\pi_1), \tag{1}\]
motivates us to ask a general question:  if {\it polynomial identities} among Hecke eigenvalues
of cuspforms on $\GL_{n_i}(\A_F)$ at almost all unramified places of $F$ 
remain valid for the associated representations of the Weil-Deligne
(or, Weil) group at each finite and infinite place $v$ of $F$. We do not know if
principle of functoriality alone would suffice to prove this. We take a moment to make the question precise. For this, recall that the representation ring of polynomial representations of
$\GL_{n_i}(\C)$ is the polynomial
ring $\Z[\omega^1_{n_i},\omega^2_{n_i},\cdots, \omega^{n_i}_{n_i}]$ where $\omega^a_{n_i}$ represents
the irreducible representation of $\GL_{n_i}(\C)$ on $\Lambda^a(\C^{n_i})$. Given cuspidal automorphic
representations $\Pi_1,\Pi_2,\cdots, \Pi_d$ of $\GL_{n_i}(\A_F)$, $i=1,2,\cdots, d$,  by a polynomial
identity among $\Pi_1,\Pi_2,\cdots, \Pi_d $,  we will mean an element $P$ of the polynomial ring
$\Z[\omega^{m_i}_{n_i}]$ where $1\leq i \leq d$, and $1 \leq m_i \leq n_i$ such that
the evaluation of the polynomial $P$ on the Hecke eigenvalues of $\Pi_i$ is identically zero at
almost all the  finite places $v$ of $F$ where each of the $\Pi_i$ is unramified.   
We write  a polynomial identity among the automorphic representations $\Pi_i$ as
$P(\Pi_1,\Pi_2,\cdots, \Pi_d)=0 $ although we need to keep in mind that the ``variables'' in the polynomial $P$ is $\Pi_i$ together with its exterior powers $\Lambda^m(\Pi_i)$, $1 \leq m \leq n_i$. The question
is:  if $P(\Pi_1,\Pi_2,\cdots, \Pi_d)=0 $ in the sense we just defined, then $P(\Pi_{1v},\Pi_{2v},\cdots, \Pi_{dv})=0 $ as a representation
$W'_{F_v}$ (the Weil-Deligne, or Weil, group) at each finite or infinite place $v$ of $F$ where, now, 
$\Pi_{iv}$ is,  by abuse of notation,  a representation of
$W'_{F_v}$ associated by the local Langlands correspondence to the local component of $\Pi_i$ at the place $v$ of $F$.
  \end{remark}
\section{Isobaric types of $\Sym^6(\pi)$}

For an automorphic representation $\pi_1$ of $\GL_n(\A_F)$ with isobaric decomposition
\[\pi_1 = \pi_{11} \boxplus \pi_{12} \boxplus \cdots \boxplus \pi_{1\ell},\]
where $\pi_{1j}$ are irreducible  cuspidal automorphic representations of $\GL_{d_j}(\A_F)$, we call the set of un-ordered
integers $\{d_j\}$ which forms a partition of $n$ to be the isobaric type of $\pi_1$.

\begin{prop} \label{sym6}
  Let $\pi$ be a cuspidal non-CM automorphic representation of $\GL_2(\A_F)$ with central character
  $\omega: \A^\times_F/F^\times \rightarrow \C^\times$. Assume that $\Sym^i(\pi)$ are automorphic for $i \leq 6$, and that  $\Sym^6(\pi)$ is not cuspidal. Then
  the 
  isobaric type of $\Sym^6(\pi)$ is $(3,3,1)$ if $\pi$ is tetrahedral, $(4,2,1)$ if $\pi$ is octahedral, and of type $(4,3)$ otherwise. 
\end{prop}

\begin{proof}
  We will split the proof according to the different cases for  $\pi$.

\vspace{.5cm}

\noindent{\bf Tetrahedral case:}
In this case, one knows that $\Sym^3(\pi)$ is reducible,
and by Theorem 2.2.2 of Kim-Shahidi \cite{KS3},
\[\Sym^3(\pi) = \chi_1 \pi \boxplus \chi_2 \pi, \tag{1} \]
for certain Gr\"ossencharacters $\chi_1,\chi_2$ of $F$. Since,
\[\Sym^2\Sym^3 (\pi) = \Sym^6(\pi) \boxplus \omega^2 \Sym^2(\pi),\tag{2}\]
equation (1) gives:  \[ \Sym^6(\pi) \boxplus \omega^2\Sym^2(\pi) = \chi_1^2 \Sym^2(\pi) \boxplus \chi_2^2\Sym^2(\pi) \boxplus \chi_1 \chi_2 \pi \boxtimes \pi.\tag{3} \]

Since we are assuming that $\pi$ is non-CM, $\Sym^2(\pi)$ is a cuspidal automorphic representation of $\GL_3(\A_F)$,
and thus the only option for the isobaric type of $\Sym^6(\pi)$ is (3,3,1).

\vspace{.5cm}

\noindent{\bf Octahedral case:}  In this case, one knows that
$\Sym^2(\pi), \Sym^3(\pi)$ are irreducible, but
$\Sym^4(\pi)$ is reducible, and by Theorem 3.3.7(3) of Kim-Shahidi \cite{KS3},
\[\Sym^4(\pi) = \chi_1 \pi \boxplus \chi_2 \Sym^2(\pi). \tag{4}\]
Since,
\[\Lambda^2\Sym^4 (\pi) = \omega \Sym^6(\pi) \boxplus \omega^2\Sym^2(\pi), \tag{5}\]
using (4) and (5) we have,  \[ \omega \Sym^6(\pi) \boxplus \omega^2\Sym^2(\pi) = \chi_1^2 \Lambda^2(\pi) \boxplus \chi_2^2\Lambda^2\Sym^2( \pi) \boxplus \chi_1 \chi_2 \pi \boxtimes \Sym^2(\pi). \tag{6}\]
On the other hand, \[ \pi \boxtimes \Sym^2(\pi) = \Sym^3(\pi) \boxplus \omega \pi. \tag{7}\]

Using (6) and (7) we find that  the isobaric type of $\Sym^6(\pi)$ is (4,2,1).

\vspace{.5cm}
\noindent{\bf Rest of the cases when  $\Sym^6(\pi)$  is not cuspidal:} (Although one expects this case to consist exactly of icosahedral representations, this seems not known. Such automorphic representations of $\GL_2(\A_F)$  are called {\it quasi-icosahedral} in \cite{Ra}.
For our proof below, this lack of knowledge does not matter.)

By Theorem A$^{\prime}$ of
Ramakrishnan \cite{Ra}, if we are not in the above two cases,  and $\Sym^6(\pi)$  is not cuspidal, then 
$\Sym^i(\pi)$ are cuspidal for all $i \leq 5$, and
\[\Sym^5(\pi) = \chi \pi \boxtimes \Sym^2(\pi')  \tag{8}\]
where $\chi$ is a Gr\"ossencharacter on $F$, and $\pi'$ is a cuspidal automorphic representation
on $\GL_2(\A_F)$ such that $\Sym^2(\pi)$ and $\Sym^2(\pi')$ are not twist equivalent.  

We use the identity:
\[ \pi \boxtimes \Sym^5(\pi)
= \Sym^6(\pi) \boxplus \omega \Sym^4(\pi), \tag{9}\]
therefore using (8) we have,
\[ \chi \pi \boxtimes\pi \boxtimes  \Sym^2(\pi') = \Sym^6(\pi) \boxplus \omega \Sym^4(\pi), \tag{10} \]
or,\[ \chi [\omega\boxplus \Sym^2(\pi)] \boxtimes  \Sym^2(\pi') = \Sym^6(\pi) \boxplus \omega \Sym^4(\pi), \tag{11}\]
which under the assumption that $\Sym^6(\pi)$ is automorphic, is an identity of isobaric automorphic representations.

Since  $\Sym^2(\pi')$ is a cuspidal representation
on $\GL_3(\A_F)$, there is a $\GL_3$-cuspform  in the isobaric decomposition on the left hand side of the identity (11). Further,
we are forced to have a $\GL_5$-cuspform  on the left hand side of the identity (11) to account for
$\Sym^4(\pi)$ on the right (which is given to be cuspidal), so the possible isobaric types on the left hand side  of the identity (11) are (3,5) + a partition of 4.
Thus we deduce that the  isobaric types for $\Sym^6(\pi)$, an automorphic representation on $\GL_7(\A_F)$,
is 3  + a partition of 4. However, $\Sym^6(\pi)$ cannot have, in its isobaric decomposition, a representation of $\GL_1$ or a representation of $\GL_2$ as follows from the isobaric decomposition (9) above. For example, if $\Sym^6(\pi)$ had the shape $\pi_2 + \pi_4$ with $\pi_2$ an automorphic representation of $\GL_2$ and $\pi_4$ on $\GL_4$, then clearly the Rankin product of the right hand side of (9) with $\pi_2^\vee$ will have  a pole (at $s=1$) whereas
the left hand side  of the identity (11) which will be $\pi \boxtimes \Sym^5(\pi) \boxtimes \pi_2^\vee$ does not have a pole since in our case, $\Sym^5(\pi)$ is a cuspidal representation. Thus the only option for the isobaric decomposition of $\Sym^6(\pi)$ is $(4,3)$
\end{proof}


\section{Group theoretic analogue}
 \label{finite}

 In this paper we have answered the Question \ref{conj} for $(\GL_n,\GL_{n+2})$ for $n=1,2,3$ in the positive.
  Thus the first case one may say is left
  unsettled is $(\GL_4, \GL_6)$.

  In this section,  in Example \ref{exam1} we construct instances where our Question \ref{conj} has a positive answer
 (assuming strong Artin conjecture) for $(\GL_4, \GL_6)$,
a case not treated by our work,
and then in the final  remark of the section,
we construct instances where our Question \ref{conj} has a negative answer, using Calegari's work in \cite{Ca},
for $(\GL_4, \GL_6)$. We begin with some generalities on finite groups,
focusing eventually on $\GL_2(\F_q)$.

For representations $V_1$ and $V_2$ of a group $G$, define a relationship $V_1 \preceq V_2$ (to be read as $V_1$ {\bf immersed} in $V_2$)  
if for each element $g \in G$, the set of eigenvalues of the action of $g$ on $V_1$ (counted with multiplicities)
is contained in the set of
eigenvalues of $g$ acting on $V_2$ (counted with multiplicities). Thus if $V_1 \subset V_2$ as representations of $G$, then $V_1 \preceq V_2$.
If $V_1 \preceq V_2$
and $\dim(V_1)=\dim(V_2)$, then of course,
$V_1 \cong V_2$ as $G$-modules, whereas just as in Proposition 1, if $\dim(V_1)+1=\dim(V_2)$, then
also $V_1 \preceq V_2$ implies
$V_1 \subset V_2$ as $G$-modules. 
However, it is not true in general that if $V_1 \preceq V_2,$
then $V_1 \subset  V_2$ as $G$-modules as we will now see.

\begin{prop}
Let $G=\GL_2(\F_q)$. Let $C$ be an irreducible cuspidal representation of $\GL_2(\F_q)$ of dimension $(q-1)$, and
$P$ an irreducible principal series representation of $\GL_2(\F_q)$ of dimension $(q+1)$. Assume that the central character
of $C$ and $P$ are the same, which is $\omega: Z=\F_q^\times \rightarrow \C^\times$. Then,
\[ C \preceq P.\]
\end{prop}

\begin{proof}One knows that:
\begin{enumerate}
\item The restriction of $C$ to the diagonal torus $T =\F_q^\times \times \F_q^\times $ in $\GL_2(\F_q)$
  is the set of all characters with multiplicity 1 of $T$ whose restriction to the center $Z$ is $\omega$.
  These characters of $T$ are also contained in the restriction of $P$ to $T$
  (there are two characters $\{\chi_1 \times \chi_2, \chi_2 \times \chi_1 \}$
  of $T$ appearing in the restriction of $P$ to $T$ with multiplicity 2 which are the characters
  $\{\chi_1 \times \chi_2, \chi_2 \times \chi_1 \}$
  of $T$ used to define the principal series $P$).

\item The restriction of $C$ to the anisotropic torus $S =\F_{q^2}^\times$ in $\GL_2(\F_q)$
is the set of all characters with multiplicity 1 of $S$ whose restriction to the center $Z$ is $\omega$
  except that the two characters
$\{\chi, \bar{\chi}\}$  of $S$  used to define the cuspidal representation $C$ do not appear.
On the other hand,  the restriction of $P$ to $S $ is the
 set of all characters of $S$ with multiplicity 1 whose restriction to the center $Z$ is $\omega$.

    \item The restriction of $C$ to the upper triangular unipotent group $U =\F_q$ in $\GL_2(\F_q)$ is the regular representation of
      $U$ except that the trivial representation of $U$ does not appear in $C$.
      On the other hand,  the restriction of $P$ to $U $ is the regular representation of $U$, except that the
      trivial representation appears twice.
        \end{enumerate}

It follows that $C \preceq P$, with $\dim P - \dim C = 2$. \end{proof}

Now we note the following proposition whose obvious proof will be omitted. Using this proposition, and the example of  $C \preceq P$,
we get counter examples to Question \ref{conj} at the beginning of the paper. (The group $\GL_2(\F_q)$, $q \not =2, 3, 4$ has no two dimensional
irreducible (projective) representation, thus we cannot realize $C$ and $P$ as symmetric powers of a two dimensional representation of a central
cover of $\GL_2(\F_q)$.)

\begin{prop} \label{prec}
  If $G$ is a finite group realized as a Galois group of number fields $G=\Gal(E/F)$, thus any two representations of $G$,
  $V_1$ of dimension $n$
  and $V_2$ of dimension $m$,   gives rise to Artin L-functions, which  assuming the strong Artin conjecture give rise to cuspidal automorphic representations  $\pi_1$ of $\GL_n(\A_F)$ and $\pi_2$ of $\GL_m(\A_F)$. If $V_1 \preceq V_2$, then
  the Hecke eigenvalues of the automorphic representation $\pi_1$ are contained in the Hecke eigenvalues of the automorphic representation $\pi_2$.
\end{prop}

\begin{example} \label{exam1} Here is a nice example to illustrate the use of finite groups for Question \ref{conj}. The details of the example are taken from Lemma 5.1 and 5.3 of Kim \cite{Kim2} whose notation we will follow. The group $\SL_2(\F_5)$ has two 2-dimensional
  irreducible representations $\sigma, \sigma^\tau$ (favorite of representation theorists, the odd Weil representation!). These have character values in $\Q(\sqrt{5})$, and are Galois conjugate. We have $\Sym^3(\sigma) \cong \Sym^3(\sigma^\tau)$, an irreducible 4 dimensional representation of $\SL_2(\F_5)$
  extending to a cuspidal representation $C$ of $\GL_2(\F_5)$ of non-trivial central character. Further,
  we have,
  \[\Sym^5(\sigma)\cong \Sym^5(\sigma^\tau) \cong
  \Sym^2(\sigma) \otimes \sigma^\tau \cong \Sym^2(\sigma^\tau) \otimes \sigma,\]
  giving the unique irreducible representation of $\SL_2(\F_5)$ of dimension 6
  extending to a principal series representation $P$ of $\GL_2(\F_5)$ of non-trivial central character, same as that of $C$.
    In particular, for the automorphic representations $\pi_1$ of $\GL_4(\A_F)$
  associated to $\Sym^3(\sigma)$
  and for the automorphic representations $\pi_2$ of $\GL_6(\A_F)$ associated to $\Sym^5(\sigma)$, for which since
$\Sym^3(\sigma) \preceq \Sym^5(\sigma)$,
  Proposition \ref{prec} applies, constructing an instance where
  Question \ref{conj} has an affirmative answer (using the group $G=\SL_2(\F_5)$.) 
  \end{example}

\begin{remark}
By the work of Calegari, cf. \cite{Ca}, 
  irreducible representations of $\Gal(\bar{\Q}/\Q)$ of dimension 4 and 6 factoring through
  a Galois extension $E$ of $\Q$ with $\Gal(E/\Q) = S_5 \cong \PGL_2(\F_5)$  which arise by
    taking a cuspidal representation
  of $\PGL_2(\F_5)$ of dimensions $4=q-1$, and a principal series of dimension $6=q+1$,  give rise to cuspidal automorphic representations $\Pi$ and $ \Pi'$ of
  $\GL_4(\A_\Q)$ and  $\GL_6(\A_\Q)$ respectively.

  Considerations of this section  will then give a counter-example to the Question \ref{conj}
 for the pair of cuspidal automorphic representations $\pi_1=\Pi$ and $\pi_2= \Pi'$ of
  $\GL_4(\A_\Q)$ and  $\GL_6(\A_\Q)$, 
  since by Lemma \ref{cal} below, the automorphic representation $\Pi$ of $\GL_4(\A_\Q)$ does not arise 
  as $\Sym^3(\pi)$ of an automorphic representation $\pi$ of $\GL_2(\A_\Q)$.

  We leave checking that results of this section are in
  conformity with the earlier results in the paper for $q=2,3,4$ to the reader as a curious exercise!   
\end{remark}

The proof of the following lemma is due to F. Calegari.

\begin{lemma} \label{cal}
  Let $\Pi$ be an automorphic representation on $\GL_4(\A_F)$ which  is an Artin representation
  coming from the standard 4-dimensional irreducible representation of $S_5$ (realized as the
  Galois group of an extension $E/F$). Then
  $\Pi$ cannot be written as $\Sym^3(\pi)$ of an automorphic representation $\pi$ of $\GL_2(\A_F)$.
\end{lemma}
\begin{proof}
  Choose a place $v$ of $F$ such that the Frobenius conjugacy class for the extension $E/F$ in $S_5$ is the conjugacy
  class of a transposition in $S_5$, and therefore, the Hecke eigenvalues of $\Pi$ at the place $v$ is:
  \[ (1,1, 1,-1).\]

  Suppose $\Pi_v = \Sym^3(\pi_v)$, with Hecke eigenvalues of $\pi_v$ being $(\alpha_v,\beta_v)$.
  Therefore we have the equality of un-ordered quadruples
  $ (1,1, 1,-1)$ and $(\alpha_v^3, \alpha_v^2\beta_v, \alpha_v\beta^2_v, \beta_v^3)$.
  Assume without loss of generality that $\alpha_v^3 =1$, and one of  $\alpha_v^2\beta_v, \alpha_v\beta^2_v$ is also 1. Thus either $\alpha_v=\beta_v$ or $\alpha_v=-\beta_v$. Neither is an
  option 
  if the un-ordered quadruples 
  $ (1,1, 1,-1)$ and $(\alpha_v^3, \alpha_v^2\beta_v, \alpha_v\beta^2_v, \beta_v^3)$ are the same, 
  and $\alpha_v^3 =1$.
\end{proof}

The following question was posed by F. Calegari.

\begin{question} (F. Calegari)
  If the Hecke eigenvalues of an automorphic representation $\Pi$ of $\GL_{n+1}(\A_F)$ is of the form $(\alpha_v^n, \alpha_v^{n-1}\beta_v,\cdots,
  \alpha_v\beta^{n-1}_v, \beta_v^n)$ at almost all places $v$ of $F$, then is
  $\Pi = \Sym^n(\pi)$ for an automorphic form $\pi$ of $\GL_2(\A_F)$?
  \end{question}

\section{Virtually non-abelian groups} \label{conn}

In this section we prove that the group theoretic analogue of the
Question \ref{conj} has an affirmative answer as long as the group is not `virtually abelian', i.e., does not contain a finite index subgroup which is abelian.

In the following proposition, we call a connected reductive group $Q$ of type $A_1$ if its derived subgroup is $\PGL_2(\C)$ or $\SL_2(\C)$.

\begin{prop} \label{ss}
  Let $G$ be a connected reductive algebraic group over $\C$. Let 
  $\pi_1$ and $\pi_2$ be two finite dimensional representations of $G$ with $\pi_1 \preceq \pi_2$
such that
\[ \dim(\pi_2) - \dim(\pi_1) = 2.\]
Then either
\[     \pi_2  =  \pi_1+ \lambda+\mu,\]
where $\lambda,\mu$ are one dimensional representations of $G$,
or $G$ has a reductive quotient $Q$ of type $A_1$
and $\pi'_1,\pi'_2$  irreducible representations
of $Q$ of  dimensions $d $ and $d+2$ respectively ($d=0$ allowed) such that,
\begin{eqnarray*}
  \pi_1 & = & \pi+ \pi'_1,\\
    \pi_2 & = & \pi+ \pi'_2,
\end{eqnarray*}
for a finite dimensional representation $\pi$ of $G$.
  \end{prop}
\begin{proof}
  Let $T$ be a maximal torus in $G$, and $W$ its Weyl group. Clearly, $\pi_1 \preceq \pi_2$ if and only if the weights of
  $\pi_1$ for the torus $T$ are contained in the weights of
  $\pi_2$ (with multiplicity). Since weights are $W$-invariant, if $\pi_1 \preceq \pi_2$ with $\dim(\pi_2) - \dim(\pi_1) = 2$,
  we see that there is a set of two (not necessarily distinct) weights of $T$
  (that of $\pi_2-\pi_1$) which is $W$-invariant. By Lemma \ref{weyl} below, this means that
  either these are weights of $T$ invariant under $W$, hence arise from
  characters $\lambda,\mu: G \rightarrow \C^\times$,
  or the  group $G$ has a quotient $Q$ (obtained by dividing $G$ by all normal simple groups except one which is $\PGL_2(\C)$ or $\SL_2(\C)$), with a quotient $S$ of $T$ as a maximal torus in $Q$.
 Further, by the same Lemma \ref{weyl} below,
  these two characters of $T$ are pulled back of characters
  of $S$ via the map $T\rightarrow S$.  In the first case, i.e.,  when these are two characters $\lambda,\mu: G \rightarrow \C^\times$,
  the two representations of $G$,    $\pi_2$ and   $  \pi_1+ \lambda + \mu,$
  have the same characters on $T$, therefore must be isomorphic.

  In the second case, we appeal to the elementary fact that for a
  reductive group $Q$ of type $A_1$, with maximal torus $S$, 
  any set of two distinct characters of $S$ of the form $\{\chi,\chi^w\}$ where $w$
  is the unique nontrivial element of the Weyl group of $Q$, $\chi+\chi^w$ is the
  difference of  two irreducible representations  $\pi'_1,\pi'_2$ 
  of $Q$ of  dimensions $d $ and $d+2$ respectively ($d=0$ allowed), therefore,
  we have
  \[\pi_2-\pi_1 = \pi'_2-\pi'_1,\]
  as $T$-modules, and therefore 
    \[\pi_2+\pi'_1 = \pi'_2+\pi_1,\]
    as $G$-modules, and the conclusion of the proposition follows. \end{proof}

\begin{lemma} \label {weyl}
  Let $G$ be a connected reductive  group with $T$ a maximal torus and
  $W$ its Weyl group.
  Then if $\chi$ is a character of $T$ whose $W$-orbit has $\leq 2$ elements, then
  either $\chi$  is the restriction of a one dimensional representation of  $G$ to $T$, or $G$ has a quotient $Q$ of type $A_1$
    with $S$ a
  maximal torus of $Q$ which is a quotient of $T$, and $\chi$ is a character of $T$ factoring through $S$. 
\end{lemma}

\begin{proof}
  It suffices to prove the lemma for semisimple groups where it easily reduces to a simple group. The lemma for  a simple group  reduces to the assertion
  that if $G$ is a simple group which is not  $\PGL_2(\C)$ or
  $\SL_2(\C)$, and  if $\chi$ is a character of $T$ whose $W$-orbit has
  $\leq 2$ elements, then character must be trivial. For this, observe the
  well-known fact, cf. \cite{Hum}, Lemma B of section 10.3, that the stabilizer of  $\chi$ (in $W$), an element
  in the character group $X^{\star}(T)$,  which we assume without loss of generality to belong to fundamental Weyl chamber, is  generated by the simple reflections fixing $\chi$, hence in particular, it is the Weyl group $W_X$ of the associated Levi subgroup. Now $|W/W_X| \geq 3$ can be easily proved for groups $G$ of rank $\geq 2$, by an easy reduction to rank 2 where it is clear. 
\end{proof}

The following corollary follows by an application of Clifford theory (applied to the normal subgroup  $G_0$ of $G$) combined with Proposition \ref{ss} applied to 
$G_0$, we omit its proof.
\begin{cor}
  Let $G$ be an algebraic group over $\C$, with $G_0$, the connected component of identity, a non-abelian reductive group. Assume $G$ has  irreducible
  finite dimensional representations
  $\pi_1$ and $\pi_2$,
    with $\pi_1 \preceq \pi_2$ (when restricted to $G_0$)
such that
\[ \dim(\pi_2) - \dim(\pi_1) = 2,\]
and the action of $G$ on $\pi_1+\pi_2$ is faithful.
Then, both $\pi_1,\pi_2$ remain irreducible when restricted to $G_0$, and their
restriction to $G_0$
arises from a quotient $Q$ (not necessarily semi-simple)
of $G$ of type $ A_1$.
  \end{cor}

\begin{remark}
  By proposition \ref{ss}, there are no relations $\pi_1 \preceq \pi_2$ among irreducible representations of a connected simple algebraic group with
  $\dim(\pi_2) - \dim(\pi_1) = 2,$ 
  other than the obvious ones for $G=\SL_2(\C)$, and $G= \PGL_2(\C)$. Without the constraint on
  $\dim(\pi_2) - \dim(\pi_1) = 2,$ there are naturally many more representations, such
  as the pair of representations $\Lambda^k(\C^n), \Sym^k(\C^n)$ of $\GL_n(\C)$. It seems interesting to classify all possible pairs of irreducible representations $(\pi_1,\pi_2)$ with   $\pi_1 \preceq \pi_2$  for connected simple algebraic groups with $\dim \pi_2 - \dim \pi_1$, a fixed integer.
  \end{remark}

\end{document}